\documentclass[fleqn]{zzz}
\usepackage{mathrsfs}
\usepackage{xcolor}
\usepackage{tikz}
\usepackage{amsmath,amsthm,amsbsy,amssymb}
\usepackage{graphicx}
\usepackage{rotating}
\usepackage{fixmath}
\usepackage[pdftex]{hyperref}
\usepackage{soul}
\usepackage{stmaryrd}
\usepackage{relsize}
\usepackage{amssymb}
\usepackage{bm}
\usepackage{comment}
\usepackage{soul}
\usetikzlibrary{decorations}

\pgfdeclaredecoration{dashsoliddouble}{initial}{
 \state{initial}[width=\pgfdecoratedinputsegmentlength]{
 \pgfmathsetlengthmacro\lw{.3pt+.5\pgflinewidth}
 \begin{pgfscope}
 \pgfpathmoveto{\pgfpoint{0pt}{\lw}}%
 \pgfpathlineto{\pgfpoint{\pgfdecoratedinputsegmentlength}{\lw}}%
 \pgfmathtruncatemacro\dashnum{%
 round((\pgfdecoratedinputsegmentlength-3pt)/6pt)
 }
 \pgfmathsetmacro\dashscale{%
 \pgfdecoratedinputsegmentlength/(\dashnum*6pt + 3pt)
 }
 \pgfmathsetlengthmacro\dashunit{3pt*\dashscale}
 \pgfsetdash{{\dashunit}{\dashunit}}{0pt}
 \pgfusepath{stroke}
 \pgfsetdash{}{0pt}
 \pgfpathmoveto{\pgfpoint{0pt}{-\lw}}%
 \pgfpathlineto{\pgfpoint{\pgfdecoratedinputsegmentlength}{-\lw}}%
 \pgfusepath{stroke}
 \end{pgfscope}
 }
}

\hypersetup{
colorlinks = true,
urlcolor = blue,
linkcolor = black,
}

\definecolor{Mycolor2}{HTML}{e85d04}
\sethlcolor{black}
 
\newcommand{\qhyp}[5]{\,\mbox{}_{#1}\phi_{#2}\!\left(\!\!\begin{array}{c}{#3}\\[0.10cm]{#4}\end{array};{#5}\right)}

\newcommand{\Hhyp}[5]{\,\mbox{}_{#1}H_{#2}\!\left(\!\!
\begin{array}{c}{#3}\\[0.08cm] {#4}\end{array};#5\right)}

\newtheorem{thm}{Theorem}[section]

\newtheorem{rem}[thm]{Remark}

\makeatletter
\def\eqnarray{\stepcounter{equation}\let\@currentlabel=\theequation
\global\@eqnswtrue
\tabskip\@centering\let\\=\@eqncr
$$\halign to\displaywidth\bgroup\hfil\global\@eqcnt\z@
$\displaystyle\tabskip\z@{##}$&\global\@eqcnt\@ne
\hfil$\displaystyle{{}##{}}$\hfil
&\global\@eqcnt\tw@ $\displaystyle{##}$\hfil
\tabskip\@centering&\llap{##}\tabskip\z@\cr}

\def\endeqnarray{\@@eqncr\egroup
\global\advance\c@equation\m@ne$$\global\@ignoretrue}

\def\@yeqncr{\@ifnextchar [{\@xeqncr}{\@xeqncr[5pt]}}
\makeatother

\newcommand{\Z}{\mathbb{Z}} 
\newcommand{\R}{\mathbb{R}} 
\newcommand{\N}{\mathbb{N}} 

\newcommand{\CC}{{{\mathbb C}}}
\newcommand{\CCast}{{{\mathbb C}^\ast}}
\newcommand{\CCdag}{{{\mathbb C}^\dag}}

\newcommand{\expe}{{\mathrm e}}

\newcommand{\dd}{{\mathrm d}}

\usepackage{scalerel,url}
\let\svus_
\catcode`_=\active %
\def_{\ifmmode\expandafter\svus\expandafter\bgroup\expandafter\lowerit
\else\svus\fi}
\def\lowerit#1{\ThisStyle{\raisebox{-2\LMpt}{$\SavedStyle#1$}}\egroup}
\makeatletter
\AtBeginDocument{
\check@mathfonts
\fontdimen13\textfont2=3.5pt
\fontdimen14\textfont2=3.5pt
\fontdimen16\textfont2=2.5pt
\fontdimen17\textfont2=2.5pt
 }
\makeatother

\begin{document}
\renewcommand{\PaperNumber}{***}

\FirstPageHeading

\ShortArticleName{Bilateral discrete and continuous orthogonality relations in the $q^{-1}$-symmetric Askey scheme}

\ArticleName{Bilateral discrete and continuous orthogonality relations in the $q^{-1}$-symmetric Askey scheme}
\Author{Howard S. Cohl$\,^{\ast}\orcidB{}$
and Hans Volkmer$\,^{\dag}\orcidD{}$
}
\AuthorNameForHeading{H.~S.~Cohl, H.~Volkmer
}
\Address{$^\ast$ Applied and Computational 
Mathematics Division, National Institute 
of Standards 
and Technology, Gaithersburg, MD 20899-8910, USA
\URLaddressD{
\href{http://www.nist.gov/itl/math/msg/howard-s-cohl.cfm}
{http://www.nist.gov/itl/math/msg/howard-s-cohl.cfm}
}
} 
\EmailD{howard.cohl@nist.gov} 

\Address{$^\dag$ Department of Mathematical Sciences, University of Wisconsin–Milwaukee, PO Box 413,
Milwaukee, WI 53201, USA
\URLaddressD{
\href{https://uwm.edu/math/people/volkmer-hans/}
{https://uwm.edu/math/people/volkmer-hans/}
}
} 
\EmailD{volkmer@uwm.edu} 



\ArticleDates{Received~\today~in final form ????; 
Published online ????}
\Abstract{In the $q^{-1}$-symmetric Askey scheme, namely the $q^{-1}$-Askey--Wilson, continuous dual $q^{-1}$-Hahn, $q^{-1}$-Al-Salam--Chihara, continuous big $q^{-1}$-Hermite and continuous $q^{-1}$-Hermite polynomials, we compute bilateral discrete and continuous orthogonality relations. We also derive a $q$-beta integral which comes from the continuous orthogonality relation for the $q^{-1}$-Askey--Wilson polynomials. In the $q\to 1^{-}$ limit, this $q$-beta integral corresponds to a beta integral of Ramanujan-type which we present and provide two proofs for.
}


\Classification{33D45, 05A15, 42C05, 05E05, 33D15}




\vspace{-0.5cm}

\section{Preliminaries}
Recall the notion of a {\it multiset} 
which extends the definition of a set where the multiplicity of elements is allowed. This notion becomes important for 
basic hypergeometric functions, where numerator parameter entries or denominator parameter entries may be identical.
We adopt the following set 
notations: $\mathbb N_0:=\{0\}\cup
\mathbb N=\{0, 1, 2,\ldots\}$, and if 
$\mathbb C$ which represents 
the set of 
complex numbers, then 
$\CCast:=\CC\setminus\{0\}$,  
$\CCdag:=\{z\in\CCast: |z|<1\}$.
We also need 
the 
{\it $q$-shifted factorial} 
$(a;q)_n=(1-a)(1-qa)\cdots(1-q^{n-1}a)$, 
$n\in\mathbb N_0$ and 
one may define
\begin{eqnarray}
&&\hspace{-11.3cm}(a;q)_\infty:=\prod_{n=0}^\infty 
(1-aq^{n}),\label{poch.id:2}
\end{eqnarray}
where $|q|<1$. 
Define ${\bf a}:=\{a_1,\ldots,a_r\}$,
${\bf b}:=\{b_1,\ldots,b_s\}$. 
The {\it basic hypergeometric series}, which we 
will often use, is defined for
$q,z\in\CCast$ such that $|q|<1$, $s,r\in\mathbb N_0$, 
$b_j\not\in\Omega_q$, $j=1, \ldots, s$, as
 \cite[(1.10.1)]{Koekoeketal}
\begin{equation}
\qhyp{r}{s}{\bf a}
{\bf b}
{q,z}:=
{}_r\phi_s({\bf a};{\bf b};q,z)
:=\sum_{k=0}^\infty
\frac{({\bf a};q)_k}
{(q,{\bf b};q)_k}
\left((-1)^kq^{\binom k2}\right)^{1+s-r}
z^k.
\label{2.11}
\end{equation}
\noindent For $s+1>r$, ${}_{r}\phi_s$ is an entire
function of $z$, for $s+1=r$ then 
${}_{r}\phi_s$ is convergent for $|z|<1$, and 
for $s+1<r$ the series
is divergent unless it is {\it terminating} (one of the numerator parameters is $q^{-n}$ for $n\in\N_0$).

\section{Introduction}

\noindent Define the following orthogonal polynomials in the $q^{-1}$-symmetric family of orthogonal polynomials in the $q^{-1}$-Askey-scheme, namely
\begin{eqnarray}
&&\hspace{-3cm}\bm{\mathsf{p}}_n(x;a,b,c,d|q):=\bm{\mathsf{p}}_n[z;a,b,c,d|q]:=i^{-n}p_n[iz;ia,ib,ic,id|q^{-1}],\label{qiAW}\\  
&&\hspace{-3cm}\bm{\mathsf{p}}_n(x;a,b,c|q):=\bm{\mathsf{p}}_n[z;a,b,c|q]:=i^{-n}p_n[iz;ia,ib,ic|q^{-1}],\label{cdqiH}\\  
&&\hspace{-3.02cm}\bm{\mathsf{Q}}_n(x;a,b|q):=\bm{\mathsf{Q}}_n[z;a,b|q]:=i^{-n}Q_n[iz;ia,ib|q^{-1}],\label{qiASC}\\  
&&\hspace{-3cm}\bm{\mathsf{H}}_n(x;a|q):=\bm{\mathsf{H}}_n[z;a|q]:=i^{-n}H_n[iz;ia|q^{-1}],\label{cbqiH}\\  
&&\hspace{-3cm}\bm{\mathsf{H}}_n(x|q):=\bm{\mathsf{H}}_n[z|q]:=i^{-n}H_n[iz|q^{-1}],
\label{cqiH}
\end{eqnarray}
where $x=\frac12(z-z^{-1})$. The Askey--Wilson polynomials $p_n(x;a,b,c,d|q)$ and their symmetric subfamilies (the continuous dual $q$-Hahn $p_n(x;a,b,c|q)$, Al-Salam--Chihara $Q_n(x;a,b|q)$, continuous big $q$-Hermite $H_n(x;a|q)$ and continuous $q$-Hermite polynomials $H_n(x|q)$) are defined in \cite[(14.1.1), (14.3.1), (14.8.1), (14.18.1), (14.26.1)]{Koekoeketal} respectively. We refer to the polynomials \eqref{qiAW}--\eqref{cqiH}, as the $q^{-1}$-Askey--Wilson, continuous dual $q^{-1}$-Hahn, $q^{-1}$-Al-Salam--Chihara, continuous big $q^{-1}$-Hermite and continuous $q^{-1}$-Hermite polynomials respectively. 
These versions of these polynomials have the nice property that they remain symmetric in their parameters and are very easy to compute directly from terminating basic hypergeometric representations of the symmetric subfamilies of the Askey--Wilson polynomials using \cite[(1.8.7)]{Koekoeketal}
$(a;q^{-1})_n=q^{-\binom{n}{2}}(-a)^n(\frac{1}{a};q)_n$, $a\ne 0.$ Furthermore in both the $q$-symmetric Askey scheme or in the $q^{-1}$-symmetric Askey scheme, these polynomials can be obtained  simply by taking limits as $d\to c\to b\to a\to 0$.
Note that it is easily verified that the $q^{-1}$-Askey-Wilson polynomials
$p_n(x;a,b,c,d|q^{-1})$ which are simply renormalized Askey--Wilson polynomials with
parameters given by their reciprocals,
are
given by
\begin{eqnarray}
&&\hspace{-5cm}\bm{\mathsf p}_n(x;a,b,c,d|q)
=q^{-3\binom{n}{2}}(iabcd)^np_n(x;-\tfrac{i}{a},-\tfrac{i}{b},
-\tfrac{i}{c},-\tfrac{i}{d}|q).
\end{eqnarray}

One terminating basic hypergeometric series representation of these polynomials is given by 
\begin{eqnarray}
&&\hspace{-0.0cm}\bm{\mathsf{p}}_n(x;a,b,c,d|q)=q^{-3\binom{n}{2}}(-a^2bcd)^n\left(-\frac{1}{ab},-\frac{1}{ac},-\frac{1}{ad};q\right)_n\qhyp43{q^{-n},\frac{q^{n-1}}{abcd},\frac{z}{a},-\frac{1}{az}}{-\frac{1}{ab},-\frac{1}{ac},-\frac{1}{ad}}{q,q}
\label{qiAWPP1}\\
&&\hspace{1.8cm}=q^{-3\binom{n}{2}}(-abcdz)^n\left(-\frac{1}{ab},-\frac{1}{cz},-\frac{1}{dz};q\right)_n\qhyp43{q^{-n},\frac{z}{a},\frac{z}{b},-q^{1-n}cd}{-\frac{1}{ab},-q^{1-n}cz,-q^{1-n}dz}{q,q}.
\label{qiAWPP2}
\end{eqnarray}
Two basic hypergeometric series representation of the continuous dual $q^{-1}$-Hahn polynomials are given by 
\begin{eqnarray}
&&\hspace{-2.2cm}\bm{\mathsf{p}}_n(x;a,b,c|q)=q^{-\binom{n}{2}}(-a)^n(\tfrac{z}{a},-\tfrac{1}{az};q)_n\qhyp32{q^{-n},-q^{1-n}ab,-q^{1-n}ac}{-q^{1-n}az,\frac{q^{1-n}a}{z}}{q,q}\\
&&\hspace{0.20cm}=q^{-2\binom{n}{2}}(-abc)^n(-\tfrac{1}{ab},-\tfrac{1}{ac};q)_n\qhyp32{q^{-n},\frac{z}{a},-\frac{1}{az}}{-\frac{1}{ab},-\frac{1}{ac}}{q,-\frac{q^n}{bc}}.
\end{eqnarray}
A terminating basic hypergeometric representation of the $q^{-1}$-Al-Salam--Chihara polynomials is given by 
\begin{equation}
\hspace{0.3cm}\bm{\mathsf Q}(x;a,b|q)=q^{-\binom{n}{2}}(-b)^n(-\frac{1}{ab};q)_n\qhyp31{q^{-n},\frac{z}{a},-\frac{1}{az}}{-\frac{1}{ab}}{q,q^n\frac{a}{b}}.
\end{equation}
A terminating basic hypergeometric representation of the continuous big $q^{-1}$-Hermite polynomials is given by 
\begin{equation}
\hspace{0.3cm}\bm{\mathsf H}(x;a|q)=(-\tfrac{1}{a})^n\qhyp30{q^{-n},\frac{z}{a},-\frac{1}{az}}{-}{q,-q^na^2}.
\end{equation}
A terminating basic hypergeometric representation of the continuous $q^{-1}$-Hermite polynomials is given by 
\begin{equation}
\hspace{0.3cm}\bm{\mathsf H}(x|q)=z^n\qhyp11{q^{-n}}{0}{q,-\frac{q}{z^2}}.
\end{equation}

\medskip
For comparison with the similar polynomials which exist in the literature, observe that Askey's continuous $q^{-1}$-Hermite polynomials are given by 
\begin{equation}
\hspace{0.35cm}\bm{\mathsf{H}}_n(x|q)=h_n(x|q), 
\end{equation}
Ismail's $q^{-1}$-Al-Salam--Chihara polynomials are
\cite[(3.5)]{Ismail2020}, 
\begin{equation}
\hspace{0.3cm}Q_n(x;a,b):=a^n\frac{(\frac{z}{a};q)_n}{(q;q)_n}\qhyp21{q^{-n},-\frac{1}{bz}}{q^{1-n}\frac{a}{z}}{q,\frac{qb}{z}},
\end{equation}
and they are related as follows
\begin{equation}
\hspace{0.3cm}\bm{\mathsf{Q}}_n(x;a,b|q)=q^{-\binom{n}{2}}(-1)^n(q;q)_nQ_n(x;a,b).
\end{equation}
Furthermore, Ismail--Zhang--Zhou's continuous dual $q^{-1}$-Hahn polynomials defined by 
$V_n(x;{\bf t}|q)$ given in \cite[(5.2)]{IsmailZhangZhou2022}, namely
\begin{equation}
\hspace{0.35cm}V_n(x;a,b,c|q):=
\left(\frac{a}{q}\right)^n
\frac{(-\frac{q^2}{ac};q)_n}{(-\frac{q^2}{bc};q)_n}
\qhyp32{q^{-n},\frac{qz}{a},-\frac{q}{az}}{-\frac{q^2}{ab},-\frac{q^2}{ac}}{q,-\frac{q^{n+2}}{bc}},
\label{Vdef}
\end{equation}
they are related as follows
\begin{equation}
\hspace{0.35cm}\bm{\mathsf{p}}_n(x;a,b,c|q)=q^{-2\binom{n}{2}}(-bc)^n(-\tfrac{1}{ab},-\tfrac{1}{bc};q)_nV_n(x;qa,qb,qc|q).
\end{equation}

\noindent With Ismail--Zhang--Zhou's $q^{-1}$-Askey--Wilson polynomials defined by 
$p_n(x,{\bf t})$ given in \cite[(2.7)]{IsmailZhangZhou2022}, namely with ${\bf a}:=\{a,b,c,d\}$, 
\begin{equation}
\hspace{0.35cm}p_n(x,{\bf a}):=
\left(\frac{a}{q}\right)^n
\left(-\frac{q^2}{ab},-\frac{q^2}{ac},-\frac{q^2}{ad};q\right)_n
\qhyp43{q^{-n},\frac{q^{n+3}}{abcd},\frac{qz}{a},-\frac{q}{az}}{-\frac{q^2}{ab},-\frac{q^2}{ac},-\frac{q^2}{ad}}{q,q},
\label{pdef}
\end{equation}
they are related as follows
\begin{equation}
\hspace{0.35cm}\bm{\mathsf{p}}_n(x;a,b,c,d|q)=q^{-3\binom{n}{2}}(-abcd)^n\,p_n(x,qa,qb,qc,qd).
\label{qiAWrel}
\end{equation}

It will be helpful to discuss the state of the art in regards to orthogonality relations for the $q^{-1}$-symmetric families, namely the $q^{-1}$-Askey--Wilson polynomials and their symmetric subfamilies. The orthogonality (and biorthogonality) relations we discuss are all for a finite-family, namely there exists an $N\in\N_0$ such that the orthogonality relation is only valid for all $n,m\le N$. However when one considers the limit as $d\to 0$ of the $q^{-1}$-Askey--Wilson polynomials to the continuous dual $q^{-1}$-Hahn polynomials, then the orthogonality relations lead to an infinite family. However, for the remainder of this section, let's restrict ourselves to the four parameter finite family.

The first orthogonality relation found for this family is that which comes from the weight function which Askey found for the continuous $q^{-1}$-Hermite polynomials in \cite{MR1019849}. According to Ismail \& Masson \cite[(7.31)]{IsmailMasson1994}, Askey proved that these polynomials satisfy the following orthogonality relation in \cite{MR993347}, given as follows
\begin{eqnarray}
&&\hspace{-1.0cm}\int_{-\infty}^\infty \bm{\mathsf p}_m[q^x;{\bf a}|q]
\bm{\mathsf p}_n[q^x;{\bf a}|q]
\frac{(-q^{1+x}{\bf a},q^{1-x}{\bf a};q)_\infty}{(-q^{2x+1},-q^{1-2x};q)_\infty}\,\dd x\nonumber\\
&&\hspace{0.0cm}=\frac{(q,-qab,-qac,-qad,-qbc,-qbd,-qcd;q)_\infty}{(qabcd;q)_\infty}\nonumber\\
&&\hspace{2cm}\times q^{-6\binom{n}{2}}(-a^2b^2c^2d^2)^n
\frac{(q,-\frac{1}{ab},-\frac{1}{ac},-\frac{1}{ad},-\frac{1}{bc},-\frac{1}{bd},-\frac{1}{cd};q)_n(\frac{1}{qabcd};q)_{2n}}{(\frac{1}{qabcd};q)_n(\frac{1}{abcd};q)_{2n}}\delta_{m,n}.
\end{eqnarray}
This orthogonality relation is explicitly given in \cite[(2.9)]{IsmailZhangZhou2022}. The total mass of this orthogonality relation corresponding to the $m=n=0$ case is the famous Askey $q$-beta integral \cite{MR993347}. This orthogonality relation, by taking limits $d\to c\to b\to a\to 0$, leads to orthogonality relations for the symmetric subfamilies of the $q^{-1}$-Askey--Wilson polynomials. There also exists, in the literature, an infinite discrete bilateral orthogonality relation for the $q^{-1}$-Askey--Wilson polynomials. This orthogonality relation is given in \cite[(2.13)]{IsmailZhangZhou2022} with typographical errors and the corrected version is given 
in Theorem \ref{thm11} below.

There exists a $q$-beta integral due to Ismail--Masson \cite[(7.30)]{IsmailMasson1994} (see also cf.~\cite[(3.13)]{IsmailRahman1995}).
It is given by 
\begin{eqnarray}
&&\hspace{-0.7cm}\int_0^\infty
\frac{(iaz,-\frac{ia}{z},ibz,-\frac{ib}{z},icz,-\frac{ic}{z},idz,-\frac{id}{z};q)_\infty}{(fz,gz,-\frac{qz}{f},-\frac{qz}{g},-\frac{f}{z},-\frac{g}{z},\frac{q}{fz},\frac{q}{gz};q)_\infty}\left(1+\frac{1}{z^{2}}\right)\,\dd z
=\frac{2\pi i (\frac{ab}{q},\frac{ac}{q},\frac{ad}{q},\frac{bc}{q},\frac{bd}{q},\frac{cd}{q};q)_\infty}{(q,\frac{g}{f},\frac{qf}{g},-fg,-\frac{q}{fg},\frac{abcd}{q^3};q)_\infty},
\end{eqnarray}
where $\Im f$, $\Im g$ and $\Im(f/g)$ are not $0(\!\!\!\mod 2\pi)$ and $\Im(fg) \ne \pi(\!\!\!\mod 2\pi)$.
This $q$-beta integral leads to an orthogonality relation for continuous $q^{-1}$-Hermite polynomials, but apparently it leads to a biorthogonality relation in the 4-parameter case, see Ismail \& Masson \cite{IsmailMasson1994} (I have not yet fully verified this).

In this paper, we obtain a new continuous orthogonality relation in the 4-parameter case by starting with the infinite discrete bilateral orthogonality relation for the $q^{-1}$-Askey--Wilson polynomials given in Theorem \ref{thm11} below and using the methods described in \cite{IsmailRahman1995} in order to obtain a continuous analogue. It is given in Theorem \ref{thm02} below. The weight function for this orthogonality relation in the $d\to c\to b\to a\to 0$ case of the continuous $q^{-1}$-Hermite polynomials is given 
in \cite[Theorem 21.6.4]{Ismail:2009:CQO} for $w_2(x)$. We will also present the orthogonality relations which arise from this finite family orthogonality for the subfamilies of the $q^{-1}$-Askey--Wilson polynomials. 

\section{Infinite discrete bilateral orthogonality relations}

The $q^{-1}$-symmetric families in the $q$-Askey scheme satisfy the following infinite discrete bilateral orthogonality relations.

\begin{thm}
\label{thm11}
Let $m,n,N\in\N_0$, $q\in\CCdag$, $\alpha,a,b,c,d\in\CCast$, 
$n,m \le N$ such that  $|qabcd|< |q|^{2N}$. Then the $q^{-1}$-Askey--Wilson polynomials satisfy the following infinite discrete bilateral orthogonality relation: 
\begin{eqnarray}
&&\hspace{-2.7cm}\sum_{k=-\infty}^\infty (1+q^{2k}\alpha^2)\bm{\mathsf p}_m[q^k\alpha;{\bf a}|q]
\bm{\mathsf p}_n[q^k\alpha;{\bf a}|q]
\frac{(\frac{\alpha}{\mathbf a};q)_k}{(-q\alpha {\mathbf a};q)_k}(q abcd)^k\nonumber\\
&&\hspace{-1.7cm}=
\frac{(q,-\alpha^2,-\frac{q}{\alpha^2},-qab,-qac,-qad,-qbc,-qbd,-qcd;q)_\infty}{(-q\alpha\mathbf a,\frac{q{\mathbf a}}{\alpha},qabcd;q)_\infty}\nonumber\\
&&\hspace{-1.1cm}\times q^{-6\binom{n}{2}}(a^2b^2c^2d^2)^n
\frac{(q,-\frac{1}{ab},-\frac{1}{ac},-\frac{1}{ad},-\frac{1}{bc},-\frac{1}{bd},-\frac{1}{cd};q)_n(\frac{1}{qabcd};q)_{2n}}{(\frac{1}{qabcd};q)_n(\frac{1}{abcd};q)_{2n}}\delta_{m,n}.
\label{orthogqiAWdisc}
\end{eqnarray}
\label{thm21}
\end{thm}
\begin{proof}
Using \cite[(2.13)]{IsmailZhangZhou2022} and applying \eqref{qiAWrel} derives the infinite discrete bilateral orthogonality relation. Now define the left-hand side of \eqref{orthogqiAWdisc} as
\begin{equation}
\sum_{k=-\infty}^\infty {\sf A}_{k,m,n}(\alpha;a,b,c,d|q).
\end{equation}
Then using \eqref{qiAWPP2}, one can see that
as $k\to\pm\infty$, one has
\begin{equation}
{\sf A}_{k,n,n}(\alpha;{\bf a}|q)\sim {\sf B}_n^{\pm}(\alpha;{\bf a}|q) (q^{1-2n}abcd)^{|k|},
\end{equation}
where the constants ${\sf B}_n^{\pm}(\alpha;{\bf a}|q)$ given by
\begin{eqnarray}
&&\hspace{-5.5cm}{\sf B}_n^{+}(\alpha;{\bf a}|q):=
\vartheta(\tfrac{\alpha}{{\bf a}};q)q^{-2\binom{n}{2}}\left(\frac{q^{1-n}abcd}{\alpha}\right)^{2n}
\frac{(\frac{1}{qabcd},\frac{1}{qabcd};q)_{2n}}{(\frac{1}{qabcd},\frac{1}{qabcd};q)_{n}},\\
&&\hspace{-5.5cm}{\sf B}_n^{-}(\alpha;{\bf a}|q):=\alpha^2
\vartheta(-q\alpha{\bf a};q)q^{-6\binom{n}{2}}(abcd\alpha)^{2n}
\frac{(\frac{1}{qabcd},\frac{1}{qabcd};q)_{2n}}{(\frac{1}{qabcd},\frac{1}{qabcd};q)_{n}},
\end{eqnarray}
are independent of $k$.
Therefore the bilateral series is convergent for some $N\in\N_0$ such that $|qabcd|<|q|^{2N}$. This completes the proof.
\end{proof}

\noindent Choosing $(m,n)=(0,0)$ provides the following infinite discrete summation 
which provides the total mass for the above infinite discrete orthogonality relation.

\begin{thm}
Let $q\in\CCdag$, $\alpha,a,b,c,d\in\CCast$. Then one has the following infinite discrete bilateral summation which is equivalent to the bilateral ${}_6\psi_6$ summation, namely 
\begin{eqnarray}
&&\hspace{-0.7cm}\sum_{k=-\infty}^\infty (1+q^{2k}\alpha^2)
\frac{(\frac{\alpha}{\mathbf a};q)_k}{(-q\alpha {\mathbf a};q)_k}(q abcd)^k
=
\frac{(q,-\alpha^2,-\frac{q}{\alpha^2},-qab,-qac,-qad,-qbc,-qbd,-qcd;q)_\infty}{(-q\alpha\mathbf a,\frac{q{\mathbf a}}{\alpha},qabcd;q)_\infty}.
\end{eqnarray}
\end{thm}
\begin{proof}
Setting $(m,n)=(0,0)$ in Theorem \ref{thm21} complete the proof.
\end{proof}

\begin{thm}
\label{thm12}
Let $m,n\in\N_0$, $q\in\CCdag$, $\alpha,a,b,c\in\CCast$. Then the continuous dual $q^{-1}$-Hahn polynomials satisfy the following infinite discrete bilateral orthogonality relation: 
\begin{eqnarray}
&&\hspace{-1.0cm}\sum_{k=-\infty}^\infty (1+q^{2k}\alpha^2)\bm{\mathsf p}_m[q^k\alpha;{\bf a}|q]
\bm{\mathsf p}_n[q^k\alpha;{\bf a}|q]
\frac{(\frac{\alpha}{\mathbf a};q)_k}{(-q\alpha {\mathbf a};q)_k}q^{\binom{k}{2}}(-q\alpha abc)^k\nonumber\\
&&\hspace{0.0cm}=
\frac{(q,-\alpha^2,-\frac{q}{\alpha^2},-qab,-qac,-qbc;q)_\infty}{(-q\alpha\mathbf a,\frac{q{\mathbf a}}{\alpha};q)_\infty}
q^{-4\binom{n}{2}}
\left(\frac{a^2b^2c^2}{q}\right)^n
(q,-\frac{1}{ab},-\frac{1}{ac},-\frac{1}{bc};q)_n\delta_{m,n}.
\label{orthogcdqiHdisc}
\end{eqnarray}
\end{thm}
\begin{proof}
Start with \eqref{orthogqiAWdisc} and taking the limit as $d\to 0$ completes the proof.
\end{proof}

\begin{thm}
\label{thm13}
Let $m,n\in\N_0$, $q\in\CCdag$, $\alpha,a,b\in\CCast$. Then the $q^{-1}$-Al-Salam--Chihara polynomials satisfy the following infinite discrete bilateral orthogonality relation: 
\begin{eqnarray}
&&\hspace{-5.3cm}\sum_{k=-\infty}^\infty (1+q^{2k}\alpha^2)\bm{\mathsf Q}_m[q^k\alpha;{\bf a}|q]
\bm{\mathsf Q}_n[q^k\alpha;{\bf a}|q]
\frac{(\frac{\alpha}{\mathbf a};q)_k}{(-q\alpha {\mathbf a};q)_k}q^{2\binom{k}{2}}(q\alpha^2 ab)^k\nonumber\\
&&\hspace{-4.4cm}=
\frac{(q,-\alpha^2,-\frac{q}{\alpha^2},-qab;q)_\infty}{(-q\alpha\mathbf a,\frac{q{\mathbf a}}{\alpha};q)_\infty}q^{-2\binom{n}{2}}
\left(\frac{ab}{q}\right)^n
(q,-\frac{1}{ab};q)_n\delta_{m,n}.
\label{orthogcdqiASCdisc}
\end{eqnarray}
\end{thm}
\begin{proof}
Start with \eqref{orthogcdqiHdisc} and taking the limit as $c\to 0$ completes the proof.
\end{proof}

\begin{thm}
\label{thm14}
Let $m,n\in\N_0$, $q\in\CCdag$, $\alpha,a\in\CCast$. Then the continuous big $q^{-1}$-Hermite polynomials satisfy the following infinite discrete bilateral orthogonality relation: 
\begin{eqnarray}
&&\hspace{-5.1cm}\sum_{k=-\infty}^\infty (1+q^{2k}\alpha^2)\bm{\mathsf H}_m[q^k\alpha;{a}|q]
\bm{\mathsf H}_n[q^k\alpha;{a}|q]
\frac{(\frac{\alpha}{a};q)_k}{(-q\alpha {a};q)_k}q^{3\binom{k}{2}}(-q\alpha^3a)^k\nonumber\\
&&\hspace{-4.2cm}=
\frac{(q,-\alpha^2,-\frac{q}{\alpha^2};q)_\infty}{(-q\alpha a,\frac{q{a}}{\alpha};q)_\infty}
\frac{q^{-\binom{n}{2}}(q;q)_n}{q^n}
\delta_{m,n}.
\label{orthogcbqiHdisc}
\end{eqnarray}
\end{thm}
\begin{proof}
Start with \eqref{orthogcdqiASCdisc} and taking the limit as $b\to 0$ completes the proof.
\end{proof}

\begin{thm}
\label{thm15}
Let $m,n\in\N_0$, $q\in\CCdag$, $\alpha\in\CCast$. Then the continuous $q^{-1}$-Hermite polynomials satisfy the following infinite discrete bilateral orthogonality relation: 
\begin{eqnarray}
&&\hspace{-1.0cm}\sum_{k=-\infty}^\infty (1+q^{2k}\alpha^2)\bm{\mathsf H}_m[q^k\alpha|q]
\bm{\mathsf H}_n[q^k\alpha|q]
q^{4\binom{k}{2}}(q\alpha^4)^k
=(q,-\alpha^2,-\frac{q}{\alpha^2};q)_\infty
\frac{q^{-\binom{n}{2}}(q;q)_n}{q^n}\delta_{m,n}.
\label{orthogcqiHdisc}
\end{eqnarray}
\end{thm}
\begin{proof}
Start with \eqref{orthogcbqiHdisc} and taking the limit as $a\to 0$ completes the proof.
\end{proof}

\section{Infinite continuous bilateral orthogonality relations}

\medskip
\noindent I've been thinking about extensions of \cite[Theorem 21.6.4]{Ismail:2009:CQO} for $w_2(x)$ for the symmetric families with parameters $a,b,c,d$. Recall the definition of the orthogonality measure for continuous $q^{-1}$-Hermite polynomials $w_2(x)$ \cite[(21.6.4)]{Ismail:2009:CQO}, namely
\begin{equation}
\hspace{0.95cm}w_2(x)=\exp\left(-\frac{2}{\log q^{-1}}\left[\log(x+\sqrt{x^2+1})\right]^2\right).
\end{equation}
If you choose $x=\frac12(z-z^{-1})$, then the above definition reduces to 
\begin{equation}
\hspace{0.95cm}w_2(x)=\exp\left(-\frac{2\left(\log z\right)^2}{\log q^{-1}}\right),
\end{equation}
for the polynomials $h_n(x|q)$ which are 
defined in \cite[(21.2.5)]{Ismail:2009:CQO}. These polynomials
are related to the standard continuous $q^{-1}$-Hermite polynomials using \cite[(21.2.1)]{Ismail:2009:CQO}
\begin{equation}
\hspace{0.95cm}h_n(x|q)=\bm{\mathsf H}(x|q)=i^{-n}H_n[iz|q^{-1}],
\end{equation}
where $x=\frac12(z-1/z)$. In terms of the polynomials $\bm{\mathsf H}_n[z|q]$, the equivalent orthogonality relation is 
\cite[(21.7.7)]{Ismail:2009:CQO}
\begin{eqnarray}
&&\int_{-\infty}^\infty
(q^{x}+q^{-x})\bm{\mathsf H}_m[q^x|q]\bm{\mathsf H}_n[q^x|q]\exp\left(-2x^2\log q^{-1}\right)\,\dd x=q^{-\binom{n}{2}}\frac{(q;q)_n}{q^{n+\frac18}}\sqrt{\frac{2\pi}{\log q^{-1}}}\delta_{m,n},
\label{OR1}
\end{eqnarray}
which has a positive definite measure of orthogonality for $q\in(0,1)$.
Define 
\begin{equation}
\hspace{0.95cm}\omega_2(x):=\exp\left(-2x^2\log q^{-1}\right)=q^{2x^2},
\end{equation}
which follows using the laws of logarithms.
Pulling out a factor of $q^{-x}$ in the prefactor of \eqref{OR1}, one obtains the following equivalent expression for the continuous  orthogonality relation for continuous $q^{-1}$-Hermite polynomials, namely
\begin{eqnarray}
\int_{-\infty}^\infty
(1+q^{2x})\bm{\mathsf H}_m[q^x|q]\bm{\mathsf H}_n[q^x|q]q^{2x^2-x}\,\dd x=q^{-\binom{n}{2}}\frac{(q;q)_n}{q^{n+\frac18}}\sqrt{\frac{2\pi}{\log q^{-1}}}\delta_{m,n}.
\label{contI}
\end{eqnarray}
Cursory observation clearly identifies the corresponding infinite discrete bilateral orthogonality relation given by 
\begin{equation}
\sum_{k=-\infty}^\infty (1-q^{2k}\alpha^2)
H_m[q^k\alpha|q^{-1}]H_n[q^k\alpha|q^{-1}]q^{2k^2-k}\alpha^4=(q,\alpha^2,\frac{q}{\alpha^2};q)_\infty
\frac{q^{-\binom{n}{2}}(q;q)_n}{(-q)^n}\delta_{m,n},
\end{equation}
which after the replacement $\alpha\mapsto i\alpha$, one obtains
\begin{equation}
\sum_{k=-\infty}^\infty (1+q^{2k}\alpha^2)
\bm{\mathsf H}_m[q^k\alpha|q]\bm{\mathsf H}_n[q^k\alpha|q]q^{2k^2-k}\alpha^4=(q,-\alpha^2,-\frac{q}{\alpha^2};q)_\infty
\frac{q^{-\binom{n}{2}}(q;q)_n}{q^n}\delta_{m,n},
\label{OR2}
\end{equation}
which is clearly the infinite discrete bilateral analogue of \eqref{contI} for $\alpha=1$ and is corresponds to a positive definite orthogonality measure for $q\in(0,1)$. This can be seen after replacing $\alpha\mapsto 1$, 
$k\mapsto x$ and replacing the infinite bilateral sum over $k\in\Z$ with a continuous integral over $x\in\R$, or otherwise, using the methods described in \cite{IsmailRahman1995}.

\medskip
\noindent Now we can use the methods in \cite{IsmailRahman1995} to obtain precisely the correspondence between the infinite discrete bilateral orthogonality relation \eqref{OR2} and the infinite continuous orthogonality relation given in \eqref{contI}. Both orthogonality relations having already been established in \cite{Ismail:2009:CQO} and \cite{IsmailZhangZhou2022}. This leads us to the following correspondence theorem.

\begin{thm} Let $m,n\in\N_0$, $q\in\CCdag$, $\alpha\in\CCast$ and define
\begin{eqnarray}
&&\hspace{-5cm}\Psi_{m,n}(\alpha|q):=\sum_{k=-\infty}^\infty 
(1+q^{2k}\alpha^2)\bm{\mathsf H}_n[q^k\alpha|q]
\bm{\mathsf H}_m[q^k\alpha|q]q^{2k^2-k}\alpha^{4k}\nonumber\\
&&\hspace{-2.95cm}=(q,-\alpha^2,-\frac{q}{\alpha^2};q)_\infty
\frac{q^{-\binom{n}{2}}(q;q)_n}{q^n}\delta_{m,n}.
\label{bilat}
\end{eqnarray}
Then
\begin{eqnarray}
&&\hspace{-2.6cm}{\sf K}_{m,n}(\alpha|q):=\int_{-\infty}^\infty
(1+q^{2x}\alpha^2)\bm{\mathsf H}_n[q^{x}\alpha|q]
\bm{\mathsf H}_m[q^{x}\alpha|q]
q^{2x^2-x}\alpha^{4x}\,\dd x\nonumber\\
&&\hspace{-0.6cm}=\int_{0}^1 \Psi_{m,n}(q^x\alpha|q)q^{2x^2-x}\alpha^{4x}\,\dd x\nonumber\\
&&\hspace{-0.6cm}=(q;q)_\infty
\frac{q^{-\binom{n}{2}}(q;q)_n\delta_{m,n}}{q^n}
{\sf J}(\alpha|q)\nonumber\\
&&\hspace{-0.6cm}:=(q;q)_\infty
\frac{q^{-\binom{n}{2}}(q;q)_n\,\delta_{m,n}}{q^n}
\int_0^1(-q^{2x}\alpha^2,-\frac{q^{1-2x}}{\alpha^2};q)_\infty
q^{2x^2-x}\alpha^{4x}\,\dd x,
\label{corres1}
\end{eqnarray}
where
\begin{eqnarray}
&&\hspace{-3.0cm}{\sf J}(\alpha|q)=  \int_0^1(-q^{2x}\alpha^2,-\frac{q^{1-2x}}{\alpha^2};q)_\infty
q^{2x^2-x}\alpha^{4x}\,\dd x\nonumber\\
&&\hspace{-1.8cm}=\frac{1}{(q;q)_\infty}
\int_{-\infty}^\infty(1+q^{2x}\alpha^2)q^{2x^2-x}\alpha^{4x}\,\dd x=
\frac{\sqrt{2\pi}\,\alpha \exp\left(\frac{2(\log \alpha)^2}{\log q^{-1}}\right)}
{q^\frac18\sqrt{\log q^{-1}}(q;q)_\infty}.
\label{Jint}
\end{eqnarray}
\end{thm}
\begin{proof}
The infinite discrete bilateral orthogonality relation for continuous $q^{-1}$-Hermite polynomials \eqref{bilat} is known, namely \eqref{orthogcqiHdisc}, with the finite $q$-shifted factorials rewritten as infinite $q$-shifted factorials. Using 
the methods in \cite{IsmailRahman1995} one has 
\begin{eqnarray}
&&\hspace{-3.6cm}\int_{-\infty}^\infty
(1+q^{2x}\alpha^2)\bm{\mathsf H}_n[q^{x}\alpha|q]
\bm{\mathsf H}_m[q^{x}\alpha|q]
q^{2x^2-x}\alpha^{4x}\omega(x;q)\,\dd x\nonumber\\
&&\hspace{-2.6cm}=\int_{0}^1 \Psi_{m,n}(q^x\alpha|q)q^{2x^2-x}\alpha^{4x}\omega(x;q)\,\dd x\nonumber\\
&&\hspace{-2.6cm}=(q;q)_\infty
\frac{q^{-\binom{n}{2}}(q;q)_n}{q^n}
\int_0^1(-q^{2x}\alpha^2,-\frac{q^{1-2x}}{\alpha^2};q)_\infty
q^{2x^2-x}\alpha^{4x}\omega(x;q)\,\dd x,
\end{eqnarray}
where $\omega(x;q)$ is unit-periodic on 
$x\in\R$. Choosing $\omega(x;q)\equiv 1$ provides correspondence \eqref{corres1}.
The integral ${\sf J}(\alpha|q)$ and its relation with $K_{0,0}(\alpha|q)$, namely  
\begin{equation}
K_{0,0}(\alpha|q)=(q;q)_\infty {\sf J}(\alpha|q),
\label{rel1}
\end{equation}
is clear by setting $(m,n)=(0,0)$ in 
\eqref{corres1}.
We now take advantage of \eqref{rel1} to compute ${\sf J}(\alpha|q)$.
By 
starting with the second integral in
\eqref{Jint}, 
make the substitution 
$q^{x}=\expe^{-y}$
completing the square in the power of the exponential, making standard substitutions, and relying
on the integral
\begin{equation}
\int_{-\infty}^\infty \expe^{-x^2}\alpha^{ax}\,\dd x=\sqrt{\pi}\exp(\tfrac14 a^2(\log \alpha)^2),
\end{equation}
if $\Re a\log\alpha<0$, which is just a re-written form of the Gaussian integral, completes the proof.
\end{proof}

\noindent An important special case of this integral is given by 
\begin{eqnarray}
&&\hspace{-1cm}{\sf J}(1|q)=\int_0^1(-q^{2x},-q^{1-2x};q)_\infty q^{2x^2-x}\,\dd x
=\frac{1}{(q;q)_\infty}
\int_{-\infty}^\infty(1+q^{2x})q^{2x^2-x}\,\dd x
=\frac{q^{-\frac18}}{(q;q)_\infty}\sqrt{\frac{2\pi}{\log q^{-1}}}.
\label{J1q}
\end{eqnarray}

\noindent However, we already know the continuous orthogonality relation corresponding to the left-hand side of \eqref{corres1} 
with $\alpha=1$, namely
\eqref{contI}. Therefore we are able to solve for the $\alpha=1$ special case of the following integrals
\begin{eqnarray}
&&\hspace{-5cm}{\sf K}_{0,0}(1|q)=\int_{-\infty}^\infty (1+q^{2x})q^{2x^2-x}\,\dd x=q^{-\frac18}\sqrt{\frac{2\pi}{\log q^{-1}}},\\
&&\hspace{-5cm}{\sf J}(1|q)=\int_0^1(-q^{2x},-q^{1-2x};q)_\infty q^{2x^2-x}\,\dd x=\frac{q^{-\frac18}}{(q;q)_\infty}\sqrt{\frac{2\pi}{\log q^{-1}}}.
\label{J1q}
\end{eqnarray}
Since we know ${\sf J}(1|q)$, we can use it to derive several items for the continuous big $q^{-1}$-Hermite polynomials. First is a correspondence result.
\begin{thm} 
\label{thm02}
Let $n,m\in\N_0$, $q\in\CCdag$, $\alpha,a\in\CCast$ and define
\begin{eqnarray}
&&\hspace{-1.5cm}\Psi_{m,n}(\alpha;a|q)
:=\sum_{k=-\infty}^\infty 
(1+q^{2k}\alpha^2)\bm{\mathsf H}_n[q^k\alpha;a|q]
\bm{\mathsf H}_m[q^k\alpha;a|q]
(-q^{k+1}\alpha a,\frac{q^{1-k}a}{\alpha};q)_\infty q^{2k^2-k}\alpha^{4k}\nonumber\\
&&\hspace{0.9cm}=\frac{(q,-\alpha^2,-\frac{q}{\alpha^2};q)_\infty}{(-q\alpha a,\frac{qa}{\alpha};q)_\infty}
\frac{q^{-\binom{n}{2}}(q;q)_n}{q^n}\delta_{m,n}.
\label{bilat2}
\end{eqnarray}
Then
\begin{eqnarray}
&&\hspace{-1.6cm}{\sf K}_{m,n}(\alpha;a|q):=\int_{-\infty}^\infty
(1+q^{2x}\alpha^2)\bm{\mathsf H}_n[q^{x}\alpha;a|q]
\bm{\mathsf H}_m[q^{x}\alpha;a|q]
(-q^{x+1}\alpha a,\frac{q^{1-x}a}{\alpha};q)_\infty 
q^{2x^2-x}\alpha^{4x}\,\dd x\nonumber\\
&&\hspace{-0.2cm}=\int_{0}^1 \Psi_{m,n}(q^x\alpha;a|q)q^{2x^2-x}\alpha^{4x}\,\dd x\nonumber\\
&&\hspace{-0.2cm}=(q;q)_\infty
\frac{q^{-\binom{n}{2}}(q;q)_n\,\delta_{m,n}}{q^n}
\int_0^1(-q^{2x}\alpha^2,-\frac{q^{1-2x}}{\alpha^2};q)_\infty
q^{2x^2-x}\alpha^{4x}\,\dd x\nonumber\\
&&\hspace{-0.2cm}=\frac{\sqrt{2\pi}\,\alpha \exp\left(\frac{2(\log \alpha)^2}{\log q^{-1}}\right)}
{q^\frac18\sqrt{\log q^{-1}}}
\frac{q^{-\binom{n}{2}}(q;q)_n\delta_{m,n}}{q^n}
.\label{corres2}
\end{eqnarray}
\end{thm}
\begin{proof}
The infinite discrete bilateral orthogonality relation for continuous big $q^{-1}$-Hermite polynomials \eqref{bilat2} is known, namely \eqref{orthogcbqiHdisc} with the finite $q$-shifted factorials rewritten as infinite $q$-shifted factorials. Using 
the methods in \cite{IsmailRahman1995} one has 
\begin{eqnarray}
&&\hspace{-2.6cm}\int_{-\infty}^\infty
(1+q^{2x}\alpha^2)\bm{\mathsf H}_n[q^{x}\alpha;a|q]
\bm{\mathsf H}_m[q^{x}\alpha;a|q]
(-q^{x+1}\alpha a,\frac{q^{1-x}a}{\alpha};q)_\infty 
q^{2x^2-x}\alpha^{4x}\omega(x;q)\,\dd x\nonumber\\
&&\hspace{-1.6cm}=\int_{0}^1 \Psi_{m,n}(q^x\alpha;a|q)q^{2x^2-x}\alpha^{4x}\omega(x;q)\,\dd x\nonumber\\
&&\hspace{-1.6cm}=(q;q)_\infty
\frac{q^{-\binom{n}{2}}(q;q)_n}{q^n}
\int_0^1(-q^{2x}\alpha^2,-\frac{q^{1-2x}}{\alpha^2};q)_\infty
q^{2x^2-x}\alpha^{4x}\omega(x;q)\,\dd x,
\end{eqnarray}
where $\omega(x;q)$ is unit-periodic on 
$x\in\R$. Choosing $\omega(x;q)\equiv 1$ 
and evaluating the integral \eqref{Jint} completes the proof.
\end{proof}

\noindent Second is a continuous analogue of the infinite discrete bilateral orthogonality relation for continuous big $q^{-1}$-Hermite polynomials. Note that the reason we are able to obtain this is because the norm of orthogonality is independent of the parameter $a$.

\begin{thm}Let $n,m\in\N_0$, $q\in\CCdag$, $a\in\CCdag$. Then
\begin{eqnarray}
&&\hspace{-0.9cm}\int_{-\infty}^\infty (1+q^{2x})\bm{\mathsf H}_n[q^x;a|q]
\bm{\mathsf H}_m[q^x;a|q]
(-q^{x+1}a,q^{1-x}a;q)_\infty q^{2x^2-x}\,\dd x
=\frac{q^{-\binom{n}{2}}(q;q)_n}{q^{n+\frac18}}\sqrt{\frac{2\pi}{\log q^{-1}}}\delta_{m,n}.
\end{eqnarray}
\end{thm}
\begin{proof}
Start with Theorem \ref{thm02} and set $\alpha=1$. In this case ${\sf J}(1|q)$ is known 
\eqref{J1q}, the orthogonality relation follows.
\end{proof}

\noindent Now we proceed directly to the correspondence theorem for the $q^{-1}$-Askey--Wilson polynomials.

\begin{thm} 
\label{thm02}
Let $n,m,N\in\N_0$, $q\in\CCdag$, $\alpha,a,b,c,d\in\CCast$, ${\bf a}$ be the multiset given by $\{a,b,c,d\}$, 
$n,m \le N$ such that  $|abcd|< |q|^{2N-1}$, and define
\begin{eqnarray}
&&\hspace{-1.2cm}\Psi_{m,n}(\alpha;{\mathbf a}|q)
:=\sum_{k=-\infty}^\infty 
(1+q^{2k}\alpha^2)\bm{\mathsf p}_n[q^k\alpha;{\mathbf a}|q]
\bm{\mathsf p}_m[q^k\alpha;{\mathbf a}|q]
(-q^{k+1}\alpha {\mathbf a},\frac{q^{1-k}}{\alpha}{\mathbf a};q)_\infty q^{2k^2-k}\alpha^{4k}\nonumber\\
&&\hspace{1.25cm}=\frac{(q,-\alpha^2,-\frac{q}{\alpha^2},-qab,-qac,-qad,-qbc,-qbd,-qcd;q)_\infty}{(qabcd;q)_\infty}\nonumber\\
&&\hspace{1.7cm}\times
\frac{q^{-6\binom{n}{2}}(-a^2b^2c^2d^2)^n(q,
-\frac{1}{ab},-\frac{1}{ac},-\frac{1}{ad},-\frac{1}{bc},-\frac{1}{bd},-\frac{1}{cd};q)_n(\frac{1}{qabcd};q)_{2n}
}{(\frac{1}{qabcd};q)_n(\frac{1}{abcd};q)_{2n}}
\delta_{m,n},
\label{bilat4AW}
\end{eqnarray}
which provides the infinite discrete bilateral orthogonality relation for the $q^{-1}$-Askey--Wilson polynomials.
Then one has the following continuous integral correspondence
\begin{eqnarray}
&&\hspace{-1.5cm}{\sf K}_{m,n}(\alpha;{\mathbf a}|q)
:=\int_{-\infty}^\infty
(1+q^{2x}\alpha^2)\bm{\mathsf p}_n[q^{x}\alpha;{\mathbf a}|q]
\bm{\mathsf p}_m[q^{x}\alpha;{\mathbf a}|q]
(-q^{x+1}\alpha {\mathbf a},\frac{q^{1-x}}{\alpha}{\mathbf a}
;q)_\infty 
q^{2x^2-x}\alpha^{4x}\,\dd x\nonumber\\
&&\hspace{0.9cm}=\int_{0}^1 \Psi_{m,n}(q^x\alpha;{\mathbf a}|q)q^{2x^2-x}\alpha^{4x}\,\dd x\nonumber\\
&&\hspace{0.9cm}=\frac{\sqrt{2\pi}\,\alpha \exp\left(\frac{2(\log \alpha)^2}{\log q^{-1}}\right)(-qab,-qac,-qad,-qbc,-qbd,-qcd;q)_\infty}
{q^\frac18\sqrt{\log q^{-1}}(qabcd;q)_\infty}
\nonumber\\
&&\hspace{1.5cm}\times
\frac{q^{-6\binom{n}{2}}(-a^2b^2c^2d^2)^n(q,
-\frac{1}{ab},-\frac{1}{ac},-\frac{1}{ad},-\frac{1}{bc},-\frac{1}{bd},-\frac{1}{cd};q)_n(\frac{1}{qabcd};q)_{2n}
}{(\frac{1}{qabcd};q)_n(\frac{1}{abcd};q)_{2n}}
\delta_{m,n}.\label{corresAWp}
\end{eqnarray}
\end{thm}
\begin{proof}
The infinite discrete bilateral orthogonality relation for $q^{-1}$-Askey--Wilson polynomials \eqref{bilat4AW} is known,  namely \eqref{orthogqiAWdisc} with the finite $q$-shifted factorials rewritten as infinite $q$-shifted factorials.  
Then using 
the methods in \cite{IsmailRahman1995} one has 
\begin{eqnarray}
&&\hspace{-1.0cm}
\int_{-\infty}^\infty
(1+q^{2x}\alpha^2)\bm{\mathsf p}_n[q^{x}\alpha;{\mathbf a}|q]
\bm{\mathsf p}_m[q^{x}\alpha;{\mathbf a}|q]
(-q^{x+1}\alpha {\mathbf a},\frac{q^{1-x}}{\alpha}{\mathbf a}
;q)_\infty 
q^{2x^2-x}\alpha^{4x}\omega(x;q)\,\dd x\nonumber\\
&&\hspace{-0.4cm}=\int_{0}^1 \Psi_{m,n}(q^x\alpha;{\mathbf a}|q)q^{2x^2-x}\alpha^{4x}\omega(x;q)\,\dd x\nonumber\\
&&\hspace{-0.4cm}=q^{-\binom{n}{2}}\frac{(q;q)_n}{q^n}\frac{(q,-q^{1-n}ab,-q^{1-n}ac,-q^{1-n}ad,-q^{1-n}bc,-q^{1-n}bd,-q^{1-n}cd,q^{2-2n}abcd;q)_\infty}{(q^{1-2n}abcd,q^{2-n}abcd;q)_\infty}\nonumber\\
&&\hspace{0.5cm}\times\int_0^1(-q^{2x}\alpha^2,-\frac{q^{1-2x}}{\alpha^2};q)_\infty q^{2x^2-x}\alpha^{4x}\omega(\alpha;q)\,\dd x.\label{corresAW}
\end{eqnarray}
where $\omega(x;q)$ is unit-periodic on 
$x\in\R$. Choosing $\omega(x;q)\equiv 1$ 
and evaluating the integral \eqref{Jint} provides the result.
Now define the first integral on the left-hand side of \eqref{corresAWp} as
\begin{equation}
\int_{-\infty}^\infty {\sf C}_{m,n}(x;\alpha;{\bf a}|q)\,\dd x.
\end{equation}
Then using \eqref{qiAWPP2}, one can see that
as $x\to\pm\infty$, one has
\begin{equation}
{\sf C}_{n,n}(x;\alpha;{\bf a}|q)\sim 
\left\{ \begin{array}{ll}
\displaystyle{\sf D}_n^{+}(\alpha;{\bf a}|q)\,q^{2x^2-(2n+1)x}\alpha^{4x}\left(\frac{q^{1-x}\,{\bf a}}{\alpha};q\right)_\infty,
& \qquad\mathrm{as}\ x\to \infty\\[10pt]
\displaystyle {\sf D}_n^{-}(\alpha;{\bf a}|q)\,q^{2x^2+(2n+1)x}\alpha^{4x}\left(-q^{x+1}\alpha{\bf a};q\right)_\infty,
& \qquad\mathrm{as}\  x\to -\infty, \nonumber
\end{array} \right.
\end{equation}
where the constants ${\sf D}_n^{\pm}(\alpha;{\bf a}|q)$ given by
\begin{eqnarray}
&&\hspace{-7.0cm}{\sf D}_n^{+}(\alpha;{\bf a}|q):=
q^{-6\binom{n}{2}}\left(\frac{abcd}{\alpha}\right)^{2n}
\frac{(\frac{1}{qabcd},\frac{1}{qabcd};q)_{2n}}{(\frac{1}{qabcd},\frac{1}{qabcd};q)_{n}},\\
&&\hspace{-7.0cm}{\sf D}_n^{-}(\alpha;{\bf a}|q):=\alpha^2
q^{-6\binom{n}{2}}(abcd\alpha)^{2n}
\frac{(\frac{1}{qabcd},\frac{1}{qabcd};q)_{2n}}{(\frac{1}{qabcd},\frac{1}{qabcd};q)_{n}},
\end{eqnarray}
are independent of $x$.
Now for $x\to \pm\infty$, set $x=\pm(m+\epsilon)$ with $|\epsilon|<1$, $m\in\N_0$, $m\to\infty$. After simplification we see that convergence as $x\to\infty$  requires $|q^{2n-1}abcd|<1$ and convergence as $x\to-\infty$ requires that $|q^{1-2n}abcd|<1$. Therefore convergence of the integral is guaranteed for $n,m\le N$ where $|qabcd|<|q|^{2N}$. This completes the proof.
\end{proof}

\noindent This leads us to a new $q$-beta integral which arises by selecting the $K_{0,0}(\alpha;{\bf a}|q)$ term in \eqref{corresAWp}.

\begin{thm}
\label{qbetaint}
Let $q\in\CCdag$, $\alpha,a,b,c,d\in\CCast$, ${\bf a}$ be the multiset given by $\{a,b,c,d\}$. Then one has the following $q$-beta integral and $|abcd|<|q|^{-1}$, namely 
\begin{eqnarray}
&&\hspace{-1.5cm}
\int_{-\infty}^\infty
(1+q^{2x}\alpha^2)
(-q^{x+1}\alpha {\mathbf a},\frac{q^{1-x}}{\alpha}{\mathbf a}
;q)_\infty 
q^{2x^2-x}\alpha^{4x}\,\dd x\nonumber\\
&&\hspace{0.9cm}=\frac{\sqrt{2\pi}\,\alpha \exp\left(\frac{2(\log \alpha)^2}{\log q^{-1}}\right)(-qab,-qac,-qad,-qbc,-qbd,-qcd;q)_\infty}
{q^\frac18\sqrt{\log q^{-1}}(qabcd;q)_\infty}.\label{qbetaCohl}
\end{eqnarray}
\end{thm}
\begin{proof}
Setting $(m,n)=(0,0)$ in \eqref{corresAWp} and the requirement for convergence corresponds to Theorem \ref{thm02} with $N=0$. This completes the proof.
\end{proof}

\noindent
We are also able to compute the $q\to 1^{-}$ limit of the above $q$-beta integral.
This is accomplished by converting all the infinite $q$-shifted factorials in Theorem \ref{qbetaint} into $q$-gamma functions which are defined by
\cite[(I.35)]{GaspRah}
\begin{equation}
\Gamma_q(x):=\frac{(q;q)_\infty(1-q)^{1-x}}{(q^x;q)_\infty},
\label{qGam}
\end{equation}
and then computing the limit as $q\to 1^{-}$ since $\lim_{q\to 1^{-}}\Gamma_q(x)=\Gamma(x)$.

\begin{thm}
\label{thm02}
Let $a,b,c,d\in\CCast$, $\Re(a+b+c+d)>-1$. Then one has the following symmetric beta integral
\begin{eqnarray}
&&\hspace{-2.0cm}\int_{-\infty}^\infty
\frac{\dd x}{\Gamma(2x,-2x,1\!+\!a\!+\! x,1\!+\!a\!-\! x,1\!+\!b\!+\! x,1\!+\!b\!-\! x,1\!+\!c\!+\! x,1\!+\!c\!-\! x,1\!+\!d\!+\! x,1\!+\!d\!-\! x)}\nonumber
\\
&&\hspace{-1cm}
=-\frac{1}{2\pi^2}
\frac{\Gamma(a+b+c+d+1)}{\Gamma(a+b+1,a+c+1,a+d+1,b+c+1,b+d+1,c+d+1)},
\label{betaint}
\end{eqnarray}
where we are using the convention that a comma delineated list of argument to the gamma function represents multiplication by separate gamma functions with their corresponding arguments. 
\end{thm}
\begin{proof}
We obtain this result in two different ways. The first proof is a direct proof of the result by expressing it as bilateral hypergeometric series and then summing it. The second proof is by starting with Theorem \ref{qbetaint} and taking the $q\to 1^{-}$ limit. \\[0.1cm]
\noindent {\it First Proof}:
Let ${\bf a}:=\{a_1,a_2,a_3,a_4\}:=\{a,b,c,d\}$, $a_j>0$, $j\in\{1,2,3,4\}$.
We define functions
\[ f_j(x;a_j):=\frac{1}{\Gamma(1+a_j+x)\Gamma(1+a_j-x)}.\]
We adopt the definition for the Fourier transform given in Erd\'elyi et al.~\cite[Chapter III]{ErdelyiTI}. The Fourier transform
\[ F_j(t;a_j):=\mathcal{F}(f_j)(t):=\int_{-\infty}^\infty \expe^{-ixt} f_j(x;a_j)\,\dd x,\]
is known \cite[(1.2)]{Ramanujan1920} (see also \cite[(3.3.5)]{ErdelyiTI})
\[ F_j(t;a_j)=\begin{cases} \frac{\left(2\cos(\tfrac12 t)\right)^{2a_j}}{\Gamma(2a_j+1)} & \text{if $-\pi<t<\pi$,}\\
0 & \text{otherwise.}\end{cases} \]
Note that
\[ \mathcal{F}(fg)= \frac{1}{2\pi} \mathcal{F}(f)\ast \mathcal{F}(g) ,\]
where $F\ast G$ denotes convolution
\[ (F\ast G)(t)=\int_{-\infty}^\infty F(s)G(t-s)\, \dd s .\]
Set 
\[ f(x):=f_1(x)f_2(x)f_3(x)f_4(x) .\]
Then 
\[ F:= \mathcal{F}(f)=\frac{1}{8\pi^3}F_1\ast F_2\ast F_3\ast F_4.\]
It follows from the definition of the convolution that $F_1\ast F_2$ and $F_3\ast F_4$ vanish outside $[-2\pi,2\pi]$.
Therefore, the even function $F(t)$ vanishes outside $[-4\pi, 4\pi]$. We expand $F(t)$ in a Fourier cosine series
\[ F(t)=\frac12\sum_{k=0}^\infty \epsilon_k w_k \cos(\tfrac14 kt) ,\]
where $\epsilon_0=1$, $\epsilon_k=2$ for $k\in \N$ is the Neumann factor, and
\[ w_k:=\frac{1}{4\pi} \int_{-4\pi}^{4\pi} F(s)\cos(\tfrac14k s)\,\dd s =\frac{1}{4\pi} \int_{-\infty}^\infty F(s)\cos(\tfrac14k s)\,\dd s.\]
For every real $\omega$ we have 
\[ \int_{-\infty}^\infty F(s)\,\expe^{i\omega s}\,\dd s =2\pi  f(\omega) .\]
Therefore,
\[ w_k = \tfrac12 f(\tfrac14 k)=\tfrac12 \prod_{j=1}^4 \frac{1}{\Gamma(1+a_j+\tfrac14k)\Gamma(1+a_j-\tfrac14k)} .\]
Consider the integral
\begin{equation}\label{eq1}
\int_{-\infty}^\infty f(x)\cos(xt)\,\dd x= F(t)=\frac12\sum_{k=0}^\infty \epsilon_k w_k \cos(\tfrac14 k t).
\end{equation}
In particular, if $t=0$ then
\[ \int_{-\infty}^\infty f(x)\,\dd x =
\frac12\sum_{k=0}^\infty \epsilon_k w_k .\]
If we differentiate \eqref{eq1} with respect to $t$ we find
\[ \int_{-\infty}^\infty f(x) x\sin (xt)\,\dd x= -F'(t)=\frac14 \sum_{k=0}^\infty \epsilon_k kw_k \sin(\tfrac14 kt).
 \]
If $t=2\pi$ one obtains
\begin{eqnarray*}
&&\hspace{-1.4cm}\int_{-\infty}^\infty \frac{f(x)}{\Gamma(2x)\Gamma(-2x)}\,\dd x = -\frac2\pi\int_{-\infty}^\infty x\sin(2\pi x) f(x)\,\dd x
=J:=-\frac1{2\pi}\sum_{k=0}^\infty \epsilon_k k w_k \sin(\tfrac12k\pi).
\end{eqnarray*}
Setting $k=4n+1$, $k=4m+3$ for $m,n\in\N_0$ we find
\[ J=-\frac{1}{4\pi}\sum_{n=0}^\infty (4n+1)f(n+\tfrac14)
+\frac{1}{4\pi} \sum_{m=0}^\infty (4m+3)f(m+\tfrac34) .\]
If we replace $m$ by $-n-1$ and note that $f$ is an even function, then this represents $J$ as a bilateral sum
\[ J=-\frac{1}{4\pi} \sum_{n=-\infty}^\infty (4n+1)f(n+\tfrac14) .\]
Using the bilateral hypergeometric function ${}_5H_5$ \cite[(16.4.16)]{NIST:DLMF}
this becomes
\[ J=-\frac{1}{4\pi}
\Hhyp55{\frac54, \frac14-a_1,\frac14-a_2,\frac14-a_3,\frac14- a_4}{ \frac14, a_1+\frac54, a_2+\frac54,a_3+\frac54,a_4+\frac54}{1}
\prod_{j=1}^4\frac{1}{\Gamma(a_j+\tfrac54)\Gamma(a_j+\tfrac34)}.\]
The bilateral hypergeometric function can be evaluated using Dougall's formula \cite[(6.1.2.5)]{Slater66}, and we obtain
\begin{eqnarray*} \hspace{-9.4cm}J=-\frac{1}{2\pi^2} \frac{\Gamma(a_1+a_2+a_3+a_4+1)}{\prod_{1\le i<j\le 4} \Gamma(1+a_i+a_j)} .
\end{eqnarray*}
This completes the first proof.
\\[0.05cm]
{\it Second proof}:~First start with Theorem \ref{qbetaint} and replace $(\alpha,a,b,c,d)\mapsto i(q^\alpha,q^a,q^b,q^c,q^d)$.
Then re-writing it in terms of $q$-gamma functions using \eqref{qGam} leads to
\begin{eqnarray*}
&&\hspace{-1.0cm}\int_{-\infty}^\infty \frac{\Gamma_q(2\alpha+2x+1)\,q^{2x^2+(4\alpha-1) x}\,\expe^{2i\pi x}\,\dd x}
{\Gamma_q(2\alpha+2x,1\!+\!a\!\pm(x+\alpha),1\!+\!b\!\pm(x+\alpha),1\!+\!c\!\pm(x+\alpha),1\!+\!d\!\pm(x+\alpha))}\nonumber\\
&&\hspace{0.0cm}
=\frac{iq^{\alpha-\frac18}\sqrt{2\pi}\exp\left(\frac{2(\log(iq^\alpha ))^2}{\log q^{-1}}\right)}{(q;q)_\infty^3(1-q)\sqrt{\log q^{-1}}}
\frac{\Gamma_q(a\!+\!b\!+\!c\!+\!d\!+\!1)}
{\Gamma_q(a\!+\!b\!+\!1,a\!+\!c\!+\!1,a\!+\!d\!+\!1,b\!+\!c\!+\!1,b\!+\!d\!+\!1,c\!+\!d\!+\!1)}.
\end{eqnarray*}
Taking the $q\to1^{-}$ with justification of interchanging the limit and the integral being provided by the first proof above.
In the limit as $q\to 1^{-}$ of the above expression, we see that all the $q$-gamma functions become gamma functions, powers of $q$ become unity and one has 
\begin{eqnarray*}
&&\hspace{-7.5cm}\lim_{q\to 1^{-}}\frac{iq^{\alpha-\frac18}\sqrt{2\pi}\exp\left(\frac{2(\log(iq^\alpha ))^2}{\log q^{-1}}\right)}{(q;q)_\infty^3(1-q)\sqrt{\log q^{-1}}}=
\sqrt{2\pi}i\expe^{-2i\pi\alpha}{\sf T}.
\end{eqnarray*}
since
\begin{eqnarray*}
&&\hspace{-5.7cm}\sqrt{\log q^{-1}}\sim \sqrt{1-q},\\
&&\hspace{-5.7cm}\exp\left(\frac{2(\log(iq^\alpha ))^2}{\log q^{-1}}\right)\sim \expe^{-2\pi i\alpha}\exp\left(\frac{\pi^2}{2\log q}\right),\quad \mbox{where},\\
&&\hspace{-5.7cm}{\sf T}:=\lim_{q\to 1^{-}}\frac{\exp\left(-\frac{\pi^2}{2\log q^{-1}}\right)}
{(q;q)_\infty^3(1-q)^\frac32}.
\end{eqnarray*}
Then one obtains
\begin{eqnarray*}
&&\hspace{-2.7cm}\int_{-\infty}^\infty
\frac{(\alpha\!+\!x)\,\expe^{2\pi i x}\,\dd x}{\Gamma(1\!+\!a\!\pm \!(\alpha\!+\!x),1\!+\!b\!\pm \!(\alpha\!+\!x),1\!+\!c\!\pm \!(\alpha\!+\!x),1\!+\!d\!\pm \!(\alpha\!+\!x))}
\\
&&\hspace{-1.7cm}
=
\frac{i\sqrt{\frac{\pi}{2}}\expe^{-2i\pi\alpha}{\sf T}\,\Gamma(a+b+c+d+1)}{\Gamma(a+b+1,a+c+1,a+d+1,b+c+1,b+d+1,c+d+1)}.
\end{eqnarray*}
Then making the substitution $y=x+\alpha$ and then replacing $y\mapsto x$ obtains the expression in terms of ${\sf T}$.
If one takes the limit as $a,b,c,d,\alpha\to 0$ of the beta integral, then using the reflection formula for the gamma function \cite[(5.5.3)]{NIST:DLMF}, the beta  
integral becomes 
\begin{eqnarray*}
&&\hspace{-10cm}\int_{-\infty}^\infty
\frac{\expe^{2i\pi x}\,\sin^4(\pi x)}{x^3}\,\dd x=\frac{i\pi^3}{4}.
\end{eqnarray*}
In fact, the similar case where $a,b,c,d\to 0$ ($\alpha\ne 0$) is equivalent to the same integral.
We now see that the value of ${\sf T}$ is given by 
${\sf T}=
1/(2\pi)^\frac32$. Note that one can also see the value of the limit ${\sf T}$ by
examining \cite[(3.13)]{OldeDaalhuis94}.
Then we write the complex exponential using Euler's formula and only the $\sin$ term contributes. Finally we use the reflection formula for gamma functions to write the sine function as a a product of two gamma functions.
The limit of the constraint $|qabcd|\mapsto |q^{a+b+c+d+1}|<1$ provides the updated constraint. This completes the second proof.\\[0.1cm]
\end{proof}

\begin{rem}
Since the beta integral \eqref{betaint} is negative, one cannot use it to build a set of orthogonal polynomials with a positive measure.
\end{rem}

\noindent Now we present the correspondence theorem for the continuous dual $q^{-1}$-Hahn polynomials.

\begin{thm} 
\label{thm02}
Let $n,m\in\N_0$, $q\in\CCdag$, $\alpha,a,b,c\in\CCast$, ${\bf a}$ be the multiset given by $\{a,b,c\}$ and define
\begin{eqnarray}
&&\hspace{-0.8cm}\Psi_{m,n}(\alpha;{\mathbf a}|q)
:=\sum_{k=-\infty}^\infty 
(1+q^{2k}\alpha^2)\bm{\mathsf p}_n[q^k\alpha;{\mathbf a}|q]
\bm{\mathsf p}_m[q^k\alpha;{\mathbf a}|q]
(-q^{k+1}\alpha {\mathbf a},\frac{q^{1-k}}{\alpha}{\mathbf a};q)_\infty q^{2k^2-k}\alpha^{4k}\nonumber\\
&&\hspace{1.35cm}=(q,-\alpha^2,-\frac{q}{\alpha^2},-qab,-qac,-qbc;q)_\infty
q^{-4\binom{n}{2}}\left(\frac{a^2b^2c^2}{q}\right)^n\left(q,
-\frac{1}{ab},-\frac{1}{ac},-\frac{1}{bc};q\right)_n\!
\delta_{m,n},
\label{bilat4cdqiH}
\end{eqnarray}
which provides the infinite discrete bilateral orthogonality relation for the 
continuous dual $q^{-1}$-Hahn polynomials.
Then one has the following continuous integral correspondence
\begin{eqnarray}
&&\hspace{-0.7cm}{\sf K}_{m,n}(\alpha;{\mathbf a}|q)
:=\int_{-\infty}^\infty
(1+q^{2x}\alpha^2)\bm{\mathsf p}_n[q^{x}\alpha;{\mathbf a}|q]
\bm{\mathsf p}_m[q^{x}\alpha;{\mathbf a}|q]
(-q^{x+1}\alpha {\mathbf a},\frac{q^{1-x}}{\alpha}{\mathbf a}
;q)_\infty 
q^{2x^2-x}\alpha^{4x}\,\dd x\nonumber\\
&&\hspace{0.2cm}=\int_{0}^1 \Psi_{m,n}(q^x\alpha;{\mathbf a}|q)q^{2x^2-x}\alpha^{4x}\,\dd x\nonumber\\
&&\hspace{0.2cm}=\frac{\sqrt{2\pi}\,\alpha \exp\left(\frac{2(\log \alpha)^2}{\log q^{-1}}\right)(-qab,-qac,-qbc;q)_\infty}
{q^\frac18\sqrt{\log q^{-1}}}
q^{-4\binom{n}{2}}\left(\frac{a^2b^2c^2}{q}\right)^n\left(q,
-\frac{1}{ab},-\frac{1}{ac},-\frac{1}{bc};q\right)_n\!
\delta_{m,n}.\label{corresAW}
\end{eqnarray}
\end{thm}
\begin{proof}
The infinite discrete bilateral orthogonality relation for continuous dual $q^{-1}$-Hahn polynomials \eqref{bilat4cdqiH} is known,  namely \eqref{orthogcdqiHdisc} with the finite $q$-shifted factorials rewritten as infinite $q$-shifted factorials.  
Then using 
the methods in \cite{IsmailRahman1995} one has 
\begin{eqnarray}
&&\hspace{-0.8cm}
\int_{-\infty}^\infty
(1+q^{2x}\alpha^2)\bm{\mathsf p}_n[q^{x}\alpha;{\mathbf a}|q]
\bm{\mathsf p}_m[q^{x}\alpha;{\mathbf a}|q]
(-q^{x+1}\alpha {\mathbf a},\frac{q^{1-x}}{\alpha}{\mathbf a}
;q)_\infty 
q^{2x^2-x}\alpha^{4x}\omega(x;q)\,\dd x\nonumber\\
&&\hspace{-0.4cm}=\int_{0}^1 \Psi_{m,n}(q^x\alpha;{\mathbf a}|q)q^{2x^2-x}\alpha^{4x}\omega(x;q)\,\dd x\nonumber\\
&&\hspace{-0.4cm}=q^{-\binom{n}{2}}\frac{(q;q)_n}{q^n}(q,-q^{1-n}ab,-q^{1-n}ac,-q^{1-n}bc;q)_\infty
\int_0^1(-q^{2x}\alpha^2,-\frac{q^{1-2x}}{\alpha^2};q)_\infty q^{2x^2-x}\alpha^{4x}\omega(\alpha;q)\,\dd x.\label{corresAW}
\end{eqnarray}
where $\omega(x;q)$ is unit-periodic on 
$x\in\R$. Choosing $\omega(x;q)\equiv 1$ 
and evaluating the integral \eqref{Jint} completes the proof.
\end{proof}

\begin{thm} 
\label{thm02}
Let $n,m\in\N_0$, $q\in\CCdag$, $\alpha,a,b\in\CCast$, ${\bf a}$ be the multiset given by $\{a,b\}$ and define
\begin{eqnarray}
&&\hspace{-1.6cm}\Psi_{m,n}(\alpha;{\mathbf a}|q)
:=\sum_{k=-\infty}^\infty 
(1+q^{2k}\alpha^2)\bm{\mathsf Q}_n[q^k\alpha;{\mathbf a}|q]
\bm{\mathsf Q}_m[q^k\alpha;{\mathbf a}|q]
(-q^{k+1}\alpha {\mathbf a},\frac{q^{1-k}}{\alpha}{\mathbf a};q)_\infty q^{2k^2-k}\alpha^{4k}\nonumber\\
&&\hspace{0.85cm}=(q,-\alpha^2,-\frac{q}{\alpha^2},-qab;q)_\infty q^{-2\binom{n}{2}}
\left(\frac{ab}{q}\right)^n\left(q,-\frac{1}{ab};q\right)_n\delta_{m,n}.
\label{bilat3}
\end{eqnarray}
Then
\begin{eqnarray}
&&\hspace{-1.9cm}{\sf K}_{m,n}(\alpha;{\mathbf a}|q)
:=\int_{-\infty}^\infty
(1+q^{2x}\alpha^2)\bm{\mathsf Q}_n[q^{x}\alpha;{\mathbf a}|q]
\bm{\mathsf Q}_m[q^{x}\alpha;{\mathbf a}|q]
(-q^{x+1}\alpha {\mathbf a},\frac{q^{1-x}}{\alpha}{\mathbf a}
;q)_\infty 
q^{2x^2-x}\alpha^{4x}\,\dd x\nonumber\\
&&\hspace{-0.5cm}=\int_{0}^1 \Psi_{m,n}(q^x\alpha;{\mathbf a}|q)q^{2x^2-x}\alpha^{4x}\,\dd x\nonumber\\
&&\hspace{-0.5cm}=\frac{\sqrt{2\pi}\,\alpha \exp\left(\frac{2(\log \alpha)^2}{\log q^{-1}}\right)(-qab;q)_\infty}
{q^\frac18\sqrt{\log q^{-1}}}
q^{-\binom{n}{2}}\left(\frac{ab}{q}\right)^n\left(q,-\frac{1}{ab};q\right)_n\delta_{m,n}.\label{corres2}
\end{eqnarray}
\end{thm}
\begin{proof}
The infinite discrete bilateral orthogonality relation for $q^{-1}$-Al-Salam--Chihara polynomials \eqref{bilat3} is known,  namely \eqref{orthogcdqiASCdisc} with the finite $q$-shifted factorials rewritten as infinite $q$-shifted factorials.  
Then using 
the methods in \cite{IsmailRahman1995} one has 
\begin{eqnarray}
&&\hspace{-1.8cm}
\int_{-\infty}^\infty
(1+q^{2x}\alpha^2)\bm{\mathsf Q}_n[q^{x}\alpha;{\mathbf a}|q]
\bm{\mathsf Q}_m[q^{x}\alpha;{\mathbf a}|q]
(-q^{x+1}\alpha {\mathbf a},\frac{q^{1-x}}{\alpha}{\mathbf a}
;q)_\infty 
q^{2x^2-x}\alpha^{4x}\omega(x;q)\,\dd x\nonumber\\
&&\hspace{-0.4cm}=\int_{0}^1 \Psi_{m,n}(q^x\alpha;{\mathbf a}|q)q^{2x^2-x}\alpha^{4x}\omega(x;q)\,\dd x\nonumber\\
&&\hspace{-0.4cm}=q^{-\binom{n}{2}}\frac{(q;q)_n}{q^n}(q,-q^{1-n}ab;q)_\infty
\int_0^1(-q^{2x}\alpha^2,-\frac{q^{1-2x}}{\alpha^2};q)_\infty q^{2x^2-x}\alpha^{4x}\omega(\alpha;q)\,\dd x.\label{corresAW}
\end{eqnarray}
where $\omega(x;q)$ is unit-periodic on 
$x\in\R$. Choosing $\omega(x;q)\equiv 1$ and evaluating the integral \eqref{Jint} completes the proof.
\end{proof}




\subsection*{Acknowledgements}
We would like to thank Mourad Ismail and Keru Zhou for valuable discussions. \\[-0.8cm]

\def\cprime{$'$} \def\dbar{\leavevmode\hbox to 0pt{\hskip.2ex \accent"16\hss}d}

\end{document}